\documentclass[final,reqno]{siamltex}
\usepackage{latexsym,amsmath,amssymb,amsfonts,mathrsfs}

\usepackage{epsf,graphicx,epsfig,color,cite,cases}
\usepackage{subfigure,graphics,multirow,marginnote,enumerate,bm}
\newcommand{\R}{{\mat R}}

\newcommand{\no}{\nonumber}
\newcommand{\be}{\begin{eqnarray}}
\newcommand{\ben}{\begin{eqnarray*}}
\newcommand{\en}{\end{eqnarray}}
\newcommand{\enn}{\end{eqnarray*}}

\newcommand{\pa}{\partial}

\newcommand{\dive}{{\rm div}}

\newcommand{\mat}{\mathbb}

\newcommand{\n}{\bm{n}}
\newcommand{\0}{\bm{0}}

\newtheorem{remark}[theorem]{Remark}
\newcommand{\dx}{\,\mathrm{d}x}

\newcommand{\ds}{\,\mathrm{d}s}
\newcommand{\dy}{\,\mathrm{d}y}
\begin{document}
\renewcommand{\theequation}{\arabic{section}.\arabic{equation}}

\title{\bf
Recovering discontinuous viscosity coefficients for inverse Stokes problems by boundary measurements}
\author{Yu Jia\thanks{School of Mathematics and Statistics, Xi'an Jiaotong University,
Xi'an, Shaanxi, 710049, China ({\tt yu.jia@stu.xjtu.edu.cn})}
\and
Chengyu Wu\thanks{School of Mathematics and Statistics, Xi'an Jiaotong University,
Xi'an, Shaanxi, 710049, China ({\tt wucy99@stu.xjtu.edu.cn})}
\and
Hao Wu\thanks{School of Mathematics and Statistics, Xi'an Jiaotong University,
Xi'an, Shaanxi, 710049, China ({\tt wuhao2022@stu.xjtu.edu.cn})}
\and
Jiaqing Yang\thanks{School of Mathematics and Statistics, Xi'an Jiaotong University,
Xi'an, Shaanxi, 710049, China ({\tt jiaq.yang@mail.xjtu.edu.cn})}
}
\date{}
\maketitle

\begin{abstract}
In this paper, we investigate the inverse Stokes problem of determining a discontinuous viscosity coefficient $\mu$ in a bounded domain $\Omega\subset\mathbb{R}^3$. By analyzing the singularity of the Dirichlet Green's functions in $H^1$-norm and constructing a specifically coupled Stokes-Brinkman system in a localized domain, we prove a global uniqueness theorem that the viscosity coefficient $\mu$ can be uniquely determined from boundary measurements.
\end{abstract}



\pagestyle{myheadings}
\thispagestyle{plain}
\markboth{Y. Jia, C. Wu, H.Wu and J. Yang}{Recovering discontinuous viscosity coefficients for ISP by boundary measurements}
\section{Introduction}
\setcounter{equation}{0}
The inverse problems in hydrodynamics have recently garnered increasing attention due to their  significant applications  in various fields, such as oceanography, biomechanics, the oil industry, and naval engineering. A detailed investigation of this topic can be found in \cite{Bennett1992,Alvarez2007,Polydorides2008,Tanaka1994,Kohr2008}. In this paper, we focus on an inverse Stokes problem involving the recovery of a discontinuous viscosity coefficient from boundary measurements. Mathematically, let $\Omega$ be a bounded $C^{2}$-domains in $\mathbb{R}^3$ filled with an incompressible viscous fluid. Given boundary data $f$, the Stokes equations for the velocity field $\bm{u}:\Omega\rightarrow\mathbb{R}^3$ and the pressure field $p:\Omega\rightarrow\mathbb{R}$ are described as
\be
\begin{cases}\label{model1}
-{\rm{div}}(\mu(x)\mathcal{D}(\bm{u}))+\nabla p=0&\quad\text{ in }\Omega,\\
{\rm{div}}\bm{u}=0&\quad\text{ in }\Omega,\\
\bm{u}=\bm{f}&\quad\text{ on }\partial\Omega,
\end{cases}
\en
where $\mu$ is the viscosity coefficient . The strain tensor and the stress tensor associated with the flow $(\bm{u}, p)$ are given by
\ben
&&\mathcal{D}(\bm{u})=\frac{1}{2}([\nabla\bm{u}]+[\nabla\bm{u}]^{\top}),\\
&&T(\bm{u},p)=-p\mathbf{I}_{3}+\mu \mathcal{D}(\bm{u}),
\enn
respectively. Here, $\mathbf{I}_{3}$ denotes the identity matrix, and $\nabla\bm{u}$ is the matrix function whose $j$th column represents the gradient of the $j$th component of $\bm{u}$. Since ${\rm{div}}\bm{u}=0$, the vector $\bm{f}$ must satisfy the standard flux compatibility condition:
\be
\label{eq:condition1}
\int_{\partial\Omega}\bm{f}\cdot\bm{n}\ds=0,
\en
where $\bm{n}$ the outward normal vector on $\partial\Omega$. 

The well-posedness of Problem \eqref{model1} has been studied using either the variational approach or the integral equation method in the associated function spaces (see, e.g., \cite{Temam1984,Girault1986}). We can thus define the Dirichlet-to-Neumann (D-to-N) map as
\be
\label{eq:model.3}
\Lambda(\bm{f}):=t(\bm{u},p)\quad\text{ on }\partial\Omega,
\en
where $t(\bm{u},p):=T(\bm{u},p)\bm{n}$, and $\Lambda$ is considered to act between the space $\bm{H}^{1/2}_0(\partial\Omega)$ and $\bm{H}^{-1/2}(\partial\Omega)$, with $\bm{H}^{1/2}_0(\partial\Omega)$ defined as 
\ben
\bm{H}^{1/2}_{0}(\partial\Omega):=\lbrace \bm{f}\in\bm{H}^{1/2}(\partial\Omega): \int_{\partial\Omega}\bm{f}\cdot\bm{n}\ds=0\rbrace.
\enn
Furthermore, we define the following Cauchy data set
\ben
\mathcal{C}_{\mu}:=\{(\bm{f},\Lambda(\bm{f})):\bm{f}\in \bm{H}_0^{1/2}(\pa \Omega)\}.
\enn
Throughout this paper, we assume that $\mu$ is a discontinuous viscosity coefficient of the form
\be\label{model1-1}
\mu(x):=\mu_0\cdot\chi_{\Omega\backslash\overline{D}}(x)+a(x)\cdot\chi_{\overline{D}}(x),
\en
and satisfies the following conditions: 

$\rm{\romannumeral 1)}$  $\mu_0$ is a positive constant; 

$\rm{\romannumeral 2)}$ $a(x)\in C^{0,1}(\overline{D})$ with $a(x)>0$;

$\rm{\romannumeral 3)}$ $\mu(x)|_{\partial D}\neq\mu_0$ for any $x\in \partial D$.\\
Here, $D$ is assumed to belong to the class of admissible geometries
\ben
\mathcal{D}_{ad}:=\left\{D \subset\subset\Omega: { D \rm\ is\ open, bounded},\ C^{2}{\rm\ boundary \ and }\ \Omega\backslash\overline{D}{\rm \ is \ connected}\right\}
\enn
and $\chi_{\Omega\backslash\overline{D}}(x)$ represents the characteristic function in the subdomain $\Omega\backslash\overline{D}$, defined as 1 in $\Omega\backslash\overline{D}$ and 0 elsewhere. 

The inverse Stokes problem (ISP) is to determine $\mu$ from the Cauchy data set $\mathcal{C}_{\mu}$. The following theorem presents the main result regarding the global uniqueness of 
$\mu$.
\\
\begin{theorem}
\label{theorem2}
Assume $\mu_i|_{\overline{D_i}}\in C^{8}(\overline{D_i})$ with $D_i\in\mathcal{D}_{ad}$ for $i=1,2$. If $\mathcal{C}_{\mu_1}=\mathcal{C}_{\mu_2}$, then $\mu_1(x)=\mu_2(x)$ in $\Omega$.
\end{theorem}\\

\begin{remark}{\em
If there exists the boundary data $\bm{f}$  satisfying certain a priori assumptions, Theorem \ref{theorem2} also holds for local Cauchy data defined on any non-empty open subset of $\partial\Omega$. Furthermore, for the case with a Robin boundary condition $t(\bm{u},p)-\beta \bm{u} =\bm{f}$ on $\pa \Omega$ ($\beta\geq0$), we can prove, with slight modifications, that $\mu$ is uniquely determined by the local Robin-to-Dirichlet map defined on any non-empty open subset of $\partial\Omega$, without requiring any assumptions on $\bm{f}$.}
\end{remark}

It is well known that inverse boundary problems can date back to A.P.Calder\'{o}n \cite{Calderon1980}, who investigated the unique determination of $\sigma$ in the conductivity equation $\dive(\sigma(x)\triangledown u(x))=0$ from the Dirichlet-to-Neumann map. Significant advancements have been made in this direction since then; see fundamental papers like \cite{Kohn1984,Sylvester1987,Isakov2006,Astala2006,Nakamura1994,Pichler2018,Caro2014} for results on determining smooth coefficients in  $schr{\ddot{o}}dinger$ equations or wave equations. Furthermore, we refer to \cite{Kohn1985,Imanuvilov2010,Kenig2007,Isakov1988,Isakov2007} and the related references cited therein for results on determining discontinuous coefficients or using partial boundary measurements. However, uniqueness results for the Stokes and Navier-Stokes equations remain relatively scarce, both theoretically and numerically. It was shown in \cite{Heck2006} by Heck \emph{et al.} that a smooth viscosity coefficient $\mu(x)$ in the Stokes equations can be uniquely determined from the boundary data by constructing exponentially growing solutions connecting with $\overline{\partial}$-method in $\mathbb{R}^3$. This result was later extended by Li \emph{et al.} to the case of Navier-Stokes equations \cite{Li2007} using linearization method.  In the two-dimensional case, Imanuvilov \emph{et al.} in \cite{Imanuvilov2015} and Lai \emph{et al.} in\cite{Lai2015} established similar identifiability results for $\mu(x)$ in the Stokes and Navier-Stokes equations, respectively, using different measurements data. Furthermore, if an unknown solids is immersed in an unbounded  homogeneous fluid, the inverse problem reduces to determining the shape and location of the solid. In this case, we refer to \cite{Alves2007, Alvarez2005,Heck2007} for related works. 

Note that all aforementioned works rely heavily on the assumption that $\mu(x)$ is smooth in $\Omega$. Additionally, it is essential to impose the a priori information 
\be
\label{eq:condiation.1}
\partial^{\alpha}\mu_1(x)=\partial^{\alpha}\mu_2(x)\quad\forall x\in\partial \Omega, \  |\alpha|\leq 1
\en
for the viscosity coefficients $\mu$. For a detailed discussion in both two-dimensional and three-dimensional cases, we refer to \cite{Heck2006,Li2007,Lai2015}. If $\Omega \in \mathbb{R}^3$ is a bounded convex domain with non-vanishing Gauss curvature on $\partial \Omega$, it was shown in \cite{Heck2006} that the uniqueness can be proved without the assumption \eqref{eq:condiation.1}. 
 
 In this paper, we investigate the inverse Stokes problem of recovering a discontinuous viscosity coefficient $\mu(x)$ from boundary measurements $\mathcal{C}_{\mu}$. A novel and simple approach will be proposed, where the key technique is to analyze the singularity of the Dirichlet Green's functions in the $H^1$-norm and construct a coupled PDE-system associated with the Stokes and Brinkman equations in a localized domain. With these preparations, we can first prove that $\mu$ is uniquely determined near the boundary $\pa \Omega$ by choosing a family of special boundary data $f$ that behaves like $1/\vert x-z\vert$, with $z\in \pa \Omega$. Then, the inverse problem is reduced to determining $D$ with its physical property from $\mathcal{C}_{\mu}$ in a known fluid domain. Through an elaborate analysis conducted in a similar manner, we further prove that $D$ and $\mu(x)|_{\pa D}$ can be uniquely determined from the boundary measurements $\mathcal{C}_{\mu}$. Moreover, based on a special inequality established for two different Dirichlet Green's functions, we also show the unique determination of $\pa^{\alpha}\mu|_{\pa \Omega}$ for $|\alpha|=1$ from $\mathcal{C}_{\mu}$. These improve all previous results in e.g.\cite{Heck2006,Lai2015,Li2007,Imanuvilov2015}, where $\pa^{\alpha}\mu|_{\pa \Omega}$ has to be assumed to be known in advance for $|\alpha|\leq1$. Finally, the uniqueness of $\mu(x)|_{D}$ follows from existing results \cite{Heck2006,Lai2015} with a denseness discussion. We note that our method is independent of the choice of dimensions, and with slight modifications, the main results in Theorem \ref{theorem2} can be easily extended to the two-dimensional case.

The structure of the paper is outlined as follows. In Section 2, we introduce some useful notations and function spaces in the $H^1$-norm sense, as well as the potential theory for Stokes and Brinkman equations. In Section 3, we establish a priori estimates for the Dirichlet Green's functions and prove the well-posedness of a coupled Stokes-Brinkman system, which plays a key role in the analysis of the inverse Stokes problem and is of interest in its own right. Section 4 is devoted to the detailed proof of the main result.

\section{Preliminaries}\label{sec2}
\setcounter{equation}{0}
In this section, we first introduce some useful notations and function spaces that are employed throughout the text, and then present the potential theory for the Stokes and Brinkman equations.
\subsection{Notations}
Let $\Omega\in\mathbb{R}^3$ be a bounded domain with a $C^2$-boundary $\partial\Omega$. The standard scalar Sobolev spaces on $\Omega$ and $\partial\Omega$ are denoted by $H^s(\Omega)$ and $H^s(\partial\Omega)$, respectively, where $s\in \mathbb{R}$ and $H^0 = L^2$. The vector-valued Sobolev spaces on these domains are defined as
\ben
&&\bm{H}^s(\Omega):=[H^s(\Omega)]^3, \;\bm{H}^s(\partial\Omega):=[H^s(\partial\Omega)]^3,\qquad \qquad s\in\mathbb{R},
\enn
and the tensor-valued spaces are given by
\ben
&&\bm{H}^s(\Omega)^3:=[H^s(\Omega)]^{3\times 3}, \;\bm{H}^s(\partial\Omega)^3:=[H^s(\partial\Omega)]^{3\times 3},\quad s\in\mathbb{R}.
\enn
Additionally, for $q>1$, we define $\bm{L}^{q}(\partial\Omega):=[L^{q}(\partial\Omega)]^3$. The inner product  for matrices $A,B\in\mathbb{R}^{3\times3}$ is defined by $A:B=\sum_{i,j=1}^3 A_{ij}B_{ij}$, and the induced norm is given by $|A|=\sqrt{A:A}$. For tensor fields $A,B\in L^2(\Omega)^{3\times3}$, the inner product is defined as 
\ben
\langle A,B\rangle_{L^2(\Omega)^{3\times3}}:=\int_{\Omega}A(x):B(x)dx.
\enn
For any smooth solenoidal vector fields $\bm{u},\bm{v}$ and scalar function $p$, we have the first Green's formula 
\ben
\int_{\Omega}(-{\rm{div}}(\mu(x)D(\bm{u}))+\nabla p)\cdot\bm{v}dx=-\int_{\partial\Omega}t(\bm{u},p)\cdot\bm{v}\ds(x)+\int_{\Omega}\mu(x) D(\bm{u}):D(\bm{v})dx.
\enn

\subsection{ Potential theory for Stokes and Brinkman equations}
Let $(\mathcal{G}^{\mu,\chi^2},\Pi^{\chi^2})$ be the fundamental solution of the Brinkman system
\be\label{Brinkman}
\begin{cases}
-{\rm div}(\mu\mathcal{D}(\mathcal{G}^{\mu,\chi^2}))+\nabla \Pi^{\chi^2}+\chi^2\mathcal{G}^{\mu,\chi^2}= \delta I_{3}&\quad\text{ in }\mathbb{R}^3,\\
{\rm div}\mathcal{G}^{\mu,\chi^2}=0&\quad\text{ in }\mathbb{R}^3,\\
\end{cases}
\en
in the distribution sense, where $ \delta$ is the Dirac delta function. The components of the fundamental solution $(\mathcal{G}^{\mu,\chi^2}_{ik},\Pi_i)$ are given by(cf.\cite{Kohr2008})
\ben
&&\mathcal{G}^{\mu,\chi^2}_{ik}(x-y)=\frac{1}{4\pi\mu}\frac{\delta_{ik}}{|x-y|}A_1(R)+\frac{1}{4\pi\mu}\frac{(x_i-y_i)(x_k-y_k)}{|x-y|^3}A_2(R),\\
&&\qquad\qquad\quad\Pi_i(x-y)=\frac{x_i-y_i}{4\pi|x-y|^3},\quad i,k=1,2,3,
\enn
where $R=\frac{\chi\sqrt{2}}{\sqrt{\mu}}|x-y|$, and
\ben
A_1(z)&=&2e^{-z}(1+z^{-1}+z^{-2})-2z^{-2},\\
A_2(z)&=&-2e^{-z}(1+3z^{-1}+3z^{-2})+6z^{-2}.
\enn
 
The components of the stress and pressure tensors $\mathcal{S}^{\mu,\chi^2}$ and $\Lambda^{\mu,\chi^2}$ for the Brinkman equation are given by 

\begin{small}
\ben
\mathcal{S}^{\mu,\chi^2}_{ijk}(x-y)
&=&-\frac{1}{4\pi}\bigg\{\delta_{ik}\frac{x_j-y_j}{|x-y|^3}D_1(R)+\left(\delta_{kj}\frac{x_i-y_i}{|x-y|^3}+ \delta_{ij}\frac{x_k-y_k}{|x-y|^3}\right)\cdot D_2(R)\\
&&
+\frac{(x_i-y_i)(x_j-y_j)(x_k-y_k)}{|x-y|^5}D_3(R)\bigg\},\\
\Lambda^{\mu,\chi^2}(x-y)&=&\frac{1}{8\pi}\frac{\delta_{ik}}{|x-y|^3}(R^2-2)+\frac{3}{4\pi}\frac{(x_i-y_i)(x_k-y_k)}{|x-y|^5},
\enn
\end{small}
where 
\begin{small}
\ben
D_1(z):&=&2e^{-z}(1+3z^{-1}+3z^{-2})-6z^{-2}+1,\\
D_2(z):&=&e^{-z}(z+3+6z^{-1}+6z^{-2})-6z^{-2},\\
D_3(z):&=&e^{-z}(-2z-12-30z^{-1}-30z^{-2})+30z^{-2}.
\enn
\end{small}

We define the hydrodynamic single- and double-layer potentials for the Brinkman equation on a closed and connected Lipschitz surface $S$ and a bounded domain $\Omega$ in $\mathbb{R}^3$ as follows
\ben
&&(\bm{V}^{\mu,\chi^2}_{S})_{k}(x,\bm{\phi})=\int_{S} \mathcal{G}_{kj}^{\mu,\chi^2}(x-y)\phi_{j}(y)\ds(y),\qquad \qquad \;x\in\mathbb{R}^3\backslash S,\\
&&(P_{S}^s)(x,\bm{\phi})=\int_{S}\Pi_j(x-y)\phi_{j}(y)\ds(y),\qquad\qquad\qquad \quad \; x\in \mathbb{R}^3\backslash S,\\
&&(\bm{W}^{\mu,\chi^2}_{S})_{k}(x,\bm{\varphi})=\int_{S} \mathcal{S}_{jkl}^{\mu,\chi^2}(y-x)n_l(y)\varphi_{j}(y)\ds(y),\quad \;\;x\in\mathbb{R}^3\backslash S,\\
&&(P_{\mu,\chi^2;S}^d)(x,\bm{\varphi})=\mu\int_{S}\Lambda_{jk}^{\mu,\chi^2}(x-y)n_{k}(y)\varphi_{j}(y)\ds(y),\quad x \in \mathbb{R}^3\backslash S,
\enn
and the volume potentials are given by
\ben
&&(\bm{V}_{\Omega}^{\mu,\chi^2})_{k}(x,\bm{\phi})=\int_{\Omega} \mathcal{G}^{\mu,\chi^2}_{kj}(x-y)\phi_{j}(y)\dy,\qquad \qquad\quad \quad x \in \mathbb{R}^3,\\
&&(P_{\Omega})(x,\bm{\phi})=\int_{\Omega}\Pi_{j}(x-y)\phi_{j}(y)\dy,\qquad\qquad\qquad\qquad \quad \; x \in \mathbb{R}^3.
\enn
The boundary integral operators corresponding to the Brinkman equation are defined as 
\ben
&&(\bm{\mathcal{V}}^{\mu,\chi^2}_{S})_{k}(x,\bm{\phi})=\int_{S} \mathcal{G}_{kj}^{\mu,\chi^2}(x-y)\phi_{j}(y)\ds(y),\qquad \qquad \quad\; x\in S,\\
&&(\bm{K}^{\mu,\chi^2}_{S})_{k}(x,\bm{\phi})=\int_{S} \mathcal{S}_{jkl}^{\mu,\chi^2}(y-x)n_l(y)\phi_{j}(y)\ds(y),\quad\qquad x\in S,\\
&&(\bm{K}^{\ast,\mu,\chi^2}_{S})_{k}(x,\bm{\phi})=\int_{S} \mathcal{S}_{jkl}^{\mu,\chi^2}(x-y)n_l(x)\phi_{j}(y)\ds(y),\qquad \;x\in S,\\
&&(\bm{D}^{\mu,\chi^2}_{S})_j(x,\bm{\phi})=\int_{S}D^{\mu,\chi^2}_{jl}(x,y)\phi_{l}(y)\ds(y), \qquad \qquad\qquad \; x\in S,
\enn
where 
\ben
&&D_{jl}^{\mu,\chi^2}(x,y)=-\Lambda_{lk}^{\chi^2}(x-y)n_k(y)n_j(x)\\
 &&\qquad\qquad\qquad+\frac{\mu}{2}\left(\frac{\partial}{\partial x_j}\mathcal{S}_{lik}^{\mu,\chi^2}(y-x)+\frac{\partial}{\partial x_i}\mathcal{S}_{ljk}^{\mu,\chi^2}(y-x)\right)n_i(x)n_k(y).
 \enn
 
We refer the reader to \cite{Kohr2008,Hsiao2008} for the jump relations and continuity properties of the layer potentials in classical H{\"o}lder or standard Sobolev spaces. In the case of $\bm{L}^{q}$-spaces, due to the weak singularity of the kernel of $\bm{\mathcal{V}}^{\mu,\chi^2}_{S}$, $\bm{K}_{S}^{\mu,\chi^2}$, and $\bm{K}^{\ast,\mu,\chi^2}_{S}$ for a closed surface of class $C^{2}$, the following theorem can be derived directly from \cite{Roland2001,Roland1998}.

\begin{theorem}\label{thm-2.2}
Let $S$ be a closed surface of class $C^{2}$. Then, the boundary integral operators $\bm{\mathcal{V}}^{\mu,\chi^2}_{S}$, $\bm{K}_{S}^{\mu,\chi^2}$, and $\bm{K}^{\ast,\mu,\chi^2}_{S}$ are bounded compact mappings in $\bm{L}^{q}(S)$ for $q\in(1,+\infty)$. Moreover, the jump relations hold in $\bm{L}^{q}(S)$.
\end{theorem}

\begin{remark}\label{remark-2.3}
If we take $\frac{\chi\sqrt{2}}{\sqrt{\mu}}\rightarrow 0$, then one has $A_1(\frac{\chi\sqrt{2}}{\sqrt{\mu}}|x-y|)\rightarrow 1$, $A_2(\frac{\chi\sqrt{2}}{\sqrt{\mu}}|x-y|)\rightarrow 1$,$D_1(\frac{\chi\sqrt{2}}{\sqrt{\mu}}|x-y|)\rightarrow 0$, $D_2(\frac{\chi\sqrt{2}}{\sqrt{\mu}}|x-y|)\rightarrow 0$ and $D_3(\frac{\chi\sqrt{2}}{\sqrt{\mu}}|x-y|)\rightarrow 3$. The fundamental solution  $(\mathcal{G}^{\mu,\chi^2},\Pi)$  and the stress tensors $\mathcal{S}^{\mu,\chi^2}$ of the Brinkman equation reduce to those associated with the Stokes equation, given by 
\ben
\mathcal{G}_{ik}^{\mu}(x-y)&=&\frac{1}{4 \pi\mu}\left\{\frac{\delta_{ik}}{|x-y|}+\frac{(x_i-y_i)(x_k-y_k)}{|x-y|^{3}}\right\},\quad i,k=1,2,3,\\
\mathcal{S}_{ijk}(x-y)&=&-\frac{3}{4\pi}\frac{(x_i-y_i)(x_j-y_j)(x_k-y_k)}{|x-y|^5},\qquad i,j,k=1,2,3.
\enn
The corresponding layer potentials and boundary integral operators are denoted by $\bm{V}^{\mu}_{S},\bm{W}^{\mu}_{S}$ and $\bm{\mathcal{V}}^{\mu}_{S},\bm{K}^{\mu}_{S},\bm{K}^{\ast,\mu}_{S},\bm{D}^{\mu}_{S}$, respectively. Moreover, the potential theory discussed above is also valid for the Stokes equation.
\end{remark}

For further studies on potential theory for the Stokes and Brinkman equations in $\mathbb{R}^2$, we refer readers to \cite{Kohr2004,Kohr2012}. 

\section{Some crucial Lemmas}

In this section, we present several technical lemmas that are used frequently in the proof of the uniqueness of the inverse Stokes problem. Throughout the following discussion, we denote generic constants using lowercase $c$ and uppercase $C$, with or without subscripts, which may take different values in different contexts.
\subsection{ A priori estimates of the Dirichlet Green's functions}
Define the Dirichlet Green's function $ (G^{\mu}(x,z),r^{\mu}(x,z))$ for $z \in\Omega$, which satisfies the following equations 
\be\label{3.0}
\begin{cases}
-{\rm div}(\mu(x)\mathcal{D}(G^{\mu}(x,z)))+\nabla r^{\mu}(x,z)=\delta (x-z)I_{3} &\text{ in } \Omega,\\
{\rm div}G^{\mu}(x,z)=0 &\text{ in }\Omega\backslash\{z\},\\
G^{\mu}(x,z)=0 &\text{ on }\partial\Omega,
\end{cases}
\en
in the distribution sense. For $z_{\ast}\in\partial D$, define
\ben
z_j:=z_{\ast}+\frac{\sigma}{j}\n(z_{\ast}), \quad j=1,2,\cdots
\enn
with a sufficiently small $\sigma >0$ such that $z_j \in B(z_{\ast},\sigma)$, where $B(z_{\ast},\sigma)$ denotes a small ball with radius $\sigma$ centred at $z_{\ast}$.

 Firstly, we consider the case with a piecewise constant  coefficient $\widetilde{\mu}$, defined by
\be\no
\widetilde{\mu}(x):=\mu_0\cdot\chi_{\Omega\backslash\overline{D}}(x)+\mu_c \cdot\chi_{\overline{D}}(x), 
\en
where $\mu_c$ is a positive constant satisfying $\mu_c\neq\mu_0$. We have the following regularity estimate of the Dirichlet Green's function $ (G^{\widetilde{\mu}}(x,z_{j}),r^{\widetilde{\mu}}(x,z_{j}))$.
\\
\begin{lemma}
\label{lemma1}
$G^{\widetilde{\mu}}(x,z_{j})$  has the estimate
\be\label{lem1}
\bigg\|G^{\widetilde{\mu},D}(x,z_{j})-\bigg\lbrace\frac{\mu_c}{\mu_c+\mu_0} \mathcal{G}^{\mu_c}(x-z_j)+\frac{\mu_0}{\mu_c+\mu_0} \mathcal{G}^{\mu_0}(x-z_j)\bigg\rbrace\bigg\|_{\bm{H}^1(D)^3}\leq C
\en
for all $j\in\mathbb{N}$, where $C$ is a positive constant independent of  $j\in\mathbb{N}$.
\end{lemma}

\begin{proof}
Define 
\ben
\mathcal{G}^{\widetilde{\mu}}(x,z_j):=\begin{cases}
\mathcal{G}^{\mu_c}(x,z_j)&\text{ for } x\in \overline D,\\
\mathcal{G}^{\mu_0}(x,z_j)&\text{ for } x\in\Omega\backslash\overline{D}.
\end{cases}
\enn
for each $j\in\mathbb{N}$. Let $\bm{U}_l(x,z_j):=G_l^{\widetilde{\mu}}(x,z_j)-\mathcal{G}_l^{\widetilde{\mu}}(x,z_j)$ and $R_l(x,z_j):=r_l^{\widetilde{\mu}}(\cdot,z_j)-\Pi_l(x,z_j)$ for each $ j\in \mathbb{N}$, $l=1,2,3$. Then, it is easily founded that $(\bm{U}_l(x,z_j),R_l(x,z_j))$ satisfy the following Stokes equations
\be
\label{4.1}
\begin{cases}
-{\rm div}(\mu_0 \mathcal{D}(\bm{u}))+\nabla p=0 &\quad\text{ in }\Omega\backslash\overline{D},\\
{\rm div}\bm{u}=0 &\quad\text{ in }\Omega\backslash\overline{D},\\
-{\rm div}(\mu_c \mathcal{D}(\bm{v}))+\nabla m=0 &\quad\text{ in } D,\\
{\rm div}\bm{v}=0 &\quad\text{ in } D,
\end{cases}
\en 
with boundary conditions
\be
\label{4.2}
\begin{cases}
\bm{u}-\bm{v}=\bm{h}&\quad x\in\partial D,\\
t(\bm{u},p)-t(\bm{v},m)=\bm{g}&\quad x\in\partial D,\\
\bm{u}=\bm{k}&\quad x\in\partial\Omega,
\end{cases}
\en
where
\be
\label{4.3}
\begin{cases}
\bm{h}=\bm{h}^{(j)}_l:=\mathcal{G}_l^{\mu_c}(x,z_j)-\mathcal{G}_l^{\mu_0}(x,z_j)\in \bm{L}^q(\partial D), \quad 1<q<2,\\
\bm{g}=\bm{g}^{(j)}_l:=0\in\bm{H}^{-\frac{1}{2}}(\partial D),\\
\bm{k}=\bm{k}^{(j)}_l:=G_l^{\widetilde{\mu},D}(x,z_j)-\mathcal{G}_l^{\mu_0}(x,z_j)\in \bm{H}_0^{\frac{1}{2}}(\partial\Omega),
\end{cases}
\en
Clearly, it can be easily verified that
\be
\label{4.4}
\|\bm{h}^{(j)}_l\|_{\bm{L}^q(\partial D)}+\|\bm{g}^{(j)}_l\|_{\bm{H}^{-1/2}(\partial\Omega)}+\|\bm{k}^{(j)}_l\|_{\bm{H}_0^{1/2}(\partial D)}\leq C
\en
for all $ j\in \mathbb{N}$, $l=1,2,3$.

Now, we seek the solutions $(\bm{u},p)$ and $(\bm{v},m)$ of \eqref{4.1}-\eqref{4.2} in the form
\begin{small}
\be
\label{4.5}
&&\bm{u}(x)=\bm{V}^{\mu_0}_{\partial D}(x,\bm{\phi}_1)+\frac{1}{\mu_0}\bm{W}_{\partial D}(x,\bm{\phi}_2)+\bm{W}_{\partial\Omega}(x,\bm{\phi}_3) \qquad x\in\mathbb{R}^3\backslash(\partial D\cup\partial\Omega),\\
\label{4.7}
&&p(x)=P_{\partial D}^s(x,\bm{\phi}_1)+\frac{1}{\mu_0}P_{\mu_0,\partial D}^{d}(x,\bm{\phi}_2)
+P_{\mu_0,\partial\Omega}^{d}(x,\bm{\phi}_3)\quad x\in\mathbb{R}^3\backslash(\partial D\cup\partial\Omega),\\
\label{4.6}
&&\bm{v}(x)=\bm{V}^{\mu_c}_{\partial D}(x,\bm{\phi}_1)+\frac{1}{\mu_c}\bm{W}_{\partial D}(x,\bm{\phi}_2) \ \ \qquad\quad\qquad\qquad\quad x\in\mathbb{R}^3\backslash\partial D,\\
\label{4.8}
&&m(x)=P^s_{\partial D}(x,\bm{\phi}_1)+\frac{1}{\mu_c}P_{\mu_c,\partial D}^{d}(x,\bm{\phi}_2)\qquad\quad\qquad\qquad\quad x\in\mathbb{R}^3\backslash\partial D,
\en
\end{small}
where the density $\bm{\phi}_1=(\phi_{11},\phi_{12},\phi_{13})\in\bm{H}^{-1/2}(\partial D)$,
$\bm{\phi}_2=(\phi_{21},\phi_{22},\phi_{23})\in\bm{L}^q(\partial D)$ ($1<q<2$) and $\bm{\phi}_3=(\phi_{31},\phi_{32},\phi_{33})\in\bm{H}_0^{1/2}(\partial\Omega)$. Utilizing the jump relations and Theorem \ref{thm-2.2}, we transform \eqref{4.1}-\eqref{4.2} into the following operator equation:
\be
\label{4.9}
(\bm{A}+\bm{L})\Psi=\bm{\varrho} \quad\text{ in }\bm{H}^{-1/2}(\partial D)\times\bm{L}^q(\partial D)\times\bm{H}_0^{1/2}(\partial\Omega),
\en
where $\Psi:=(\bm{\phi}_1,\bm{\phi}_2,\bm{\phi}_3)^{\top}$ and $\bm{\varrho}:=(g,h,k)^{\top}$. The operators $\bm{A}$ and $\bm{L}$ are defined by
\ben
\bm{A}:=
\left(
\begin{array}{ccc}
-{I} & 0  & 0 \\
0  & a{I} & 0 \\
0  & 0 &-\frac{1}{2}{I}\\
\end{array}
\right),\quad
\bm{L}:=
\left(
\begin{array}{ccc}
0 & 0 &  \bm{D}^{\mu_0}_{\partial\Omega, \partial D} \\
\bm{\mathcal{V}}^{\mu_0}_{\partial D}- \bm{\mathcal{V}}^{\mu_c}_{\partial D}  & b\bm{K}_{\partial D} & \bm{K}_{\partial\Omega,\partial D}\\
\bm{\mathcal{V}}^{\mu_0}_{\partial D,\partial\Omega} & \frac{1}{\mu_0}\bm{K}_{\partial D,\partial\Omega}& \bm{K}_{\partial\Omega}\\
\end{array}
\right)
\enn
with $a:=(\frac{1}{2\mu_0}+\frac{1}{2\mu_c})$, $b:=(\frac{1}{\mu_0}-\frac{1}{\mu_c})$, where $\bm{\mathcal{V}}^{\mu}_{\partial D,\partial\Omega}$ is defined by
\ben
&&(\bm{\mathcal{V}}^{\mu}_{\partial D,\partial\Omega})_{k}(x,\bm{\phi})=\int_{\partial D} \mathcal{G}_{kj}^{\mu}(x-y)\phi_{j}(y)\ds(y),\quad x\in\partial\Omega,
\enn
and the operators $\bm{K}_{\partial\Omega,\partial D},\bm{K}_{\partial D,\partial\Omega}$ and $\bm{D}^{\mu_0}_{ \partial D,\partial\Omega}$ are defined in a similar manner. It follows from Remark \ref{remark-2.3} that all operators in $\bm{L}$ are compact in the corresponding Banach spaces. Equation \eqref{4.9} is thus of Fredholm type with index 0. Then, the existence of the solution of \eqref{4.9} follows from the uniqueness of \eqref{4.9}.

Let $(\bm{A}+\bm{L})\Psi=0$. Then, $(\bm{u},p)$ and $(\bm{v},m)$ defined by \eqref{4.5}-\eqref{4.8} with the density $\Psi$ satisfy the problem \eqref{4.1}-\eqref{4.2} with homogeneous boundary condition. By applying Green's formula, we obtain the relation
\ben
\mu_c\int_{D}\mathcal{D}(\bm{v}(x)):\mathcal{D}(\bm{v}(x))\dx
+\mu_{0}\int_{\Omega\backslash\overline{D}}\mathcal{D}(\bm{u}(x)):\mathcal{D}(\bm{u}(x))\dx=0,
\enn
which implies 
\ben
\frac{\partial{v}_{i}}{\partial x_k}+\frac{\partial{v}_{k} }{\partial x_i}&=&0,\quad \text{ in } D,\quad \quad i,k=1,2,3,\\
\frac{\partial {u}_{i}}{\partial x_k}+\frac{\partial{u}_{k} }{\partial x_i}&=&0,\quad \text{ in } \Omega\backslash\overline{D},\quad i,k=1,2,3.
\enn
Thus, $(\bm{u},p)=(\bm{0},c)$ in $\Omega\backslash\overline{D}$ and $(\bm{v},m)=(\bm{0},c)$ in $D$ with $c\in\mathbb{R}$. Furthermore, by the jump relations, it holds that 
\be
\label{4.10}
\bm{u}_{+}(x)-\bm{u}_-(x)=\bm{\phi}_3(x),\ t(\bm{u}_{+},p_{+})(x)-t(\bm{u}_-,p_{-})(x)=0 \quad x\in\partial\Omega.
\en
Here, the superscript $\bm{u}_\pm$ and $p_\pm$ indicate the limit values of the functions $\bm{u}$ and $p$ approaching $\partial\Omega$ from the exterior and interior of $\Omega$, respectively. Using again the Green formula for $\bm{u}$ in $\mathbb{R}^3\backslash\overline{\Omega}$ with \eqref{4.10}, gives
\be
\label{4.11}
\mu_{0}\int_{\mathbb{R}^3\backslash\overline{\Omega}}\mathcal{D}(\bm{u}(x)):\mathcal{D}(\bm{u}(x))\dx=-\int_{\partial {\Omega}}t(\bm{u}_{+},p_{+})\bm{u}_{+}\ds=-c\int_{\Gamma}\bm{\phi}_3\cdot\bm{n}=0.
\en
Consequently, $\bm{u}(x)=\bm{0}$ and $p(x)=c=0 $ in $\mathbb{R}^3\backslash\overline{\Omega}$. 
Moreover, from \eqref{4.10} and the form \eqref{4.5}-\eqref{4.7}, it can be verified that $\bm{\phi}_3(x)=\bm{0}$ on $\partial D$ and $\bm{u}(x)=\frac{\mu_c}{\mu_0}\bm{v}(x)$ in $\mathbb{R}^3\backslash\partial D$. 
Based on the unique continuation principle of the Stokes equation and the identities \eqref{4.10}, we can easily get $(\bm{u},p)=(\bm{v},m)=(\bm{0},0)$ in $\mathbb{R}^3\backslash\partial D$. It then follows from the jump relations that $\bm{\phi}_1(x)=\bm{\phi}_2(x)=\bm{0}$.

The above analysis shows that the homogeneous equation of \eqref{4.9} has only the trivial solution. In view of Fredholm's alternative, \eqref{4.9} has a unique solution $\Psi=(\bm{\phi}_1,\bm{\phi}_2,\bm{\phi}_3)^{\top}\in(\bm{H}^{-1/2}(\partial D)\times\bm{L}^q(\partial D)\times\bm{H}_0^{1/2}(\partial\Omega))$ satisfying the estimate 
\be\no
\|\bm{\phi}_1\|_{\bm{H}^{-1/2}(\partial D)}&+&\|\bm{\phi}_2\|_{\bm{L}^q(\partial D)}+\|\bm{\phi}_3\|_{\bm{H}_0^{1/2}(\partial\Omega)}\\\label{4.13}
&\leq & c(\|\bm{g}\|_{\bm{H}^{-1/2}(\partial D)}+\|\bm{h}\|_{\bm{L}^q(\partial D)}+\|\bm{k}\|_{\bm{H}_0^{1/2}(\partial\Omega)})\leq C.
\en
 Furthermore, utilizing again \eqref{4.9} and the embedding of $\bm{L}^q(\partial D)$ into $\bm{H}^{-1/2}(\partial D)$ with $q \geq 4/3$, we obtain
 \ben
 \|\bm{h}-a\bm{\phi}_2\|_{\bm{H}^{1/2}(\partial D)}<C.
 \enn
This, combining with \eqref{4.6} and \cite[Corollary 3.7]{David2012} gives
\ben
\bigg\|\bm{v}(x)-\frac{1}{\mu_c}\bm{W}_{\partial D}(x,\bm{\phi}_2)\bigg\|_{\bm{H}^{1}(D)}+\bigg\|\bm{W}_{\partial D}(x,\bm{h}-a\bm{\phi}_2)\bigg\|_{\bm{H}^{1}(D)}<C,
\enn
which further gives
\be
\label{4.14}
\bigg\|\bm{v}(x)-\frac{2\mu_0}{\mu_c+\mu_0}\bm{W}_{\partial D}(x,\bm{h})\bigg\|_{\bm{H}^{1}(D)}<C.
\en
Define 
\ben
\widetilde{\bm{W}}(x):&=&2\bm{W}_{\partial D}(x,\bm{h})+\mathcal{G}^{\mu_c}(x-z_j)-\mathcal{G}^{\mu_0}(x-z_j), \quad x\in D,\\
\widetilde{p}(x):&=&2P_{\mu_c;\partial D}^d(x,\bm{h})+(1-\frac{\mu_c}{\mu_0})\Pi(x,z_j), \quad x\in D.
\enn
We can verify that $(\widetilde{\bm{W}}(x),\widetilde{q}(x))$ satisfies the following Stokes problem
\be\label{4.15}
\begin{cases}
-{\rm div}\mu_c \mathcal{D}\lbrace\widetilde{\bm{W}}(x)\rbrace+\nabla \widetilde{p}(x)=0 &\quad\text{ in } D,\\
{\rm div}\widetilde{\bm{W}}(x)=0 &\quad\text{ in } D,\\
\widetilde{\bm{W}}(x)=2{W}_{\partial D}(x,\bm{h}) &\quad\text{ on }\partial D,
\end{cases}
\en
where ${W}_{\partial D}(x,\bm{h})$ is uniformly bounded in ${\bm{H}_0^{1/2}(\partial D)}$. By the well-posedness of the problem \eqref{4.15}, one has 
\be
\label{4.16}
\|2\bm{W}_{\partial D}(x,\bm{h})+ \lbrace\mathcal{G}^{\mu_c}(x-z_j)-\mathcal{G}^{\mu_0}(x-z_j)\rbrace\|_{\bm{H}^{1}(D)}<C.
\en
This, together with \eqref{4.14} and \eqref{4.16} gives the estimate \eqref{lem1}. The proof is thus complete.
\end{proof}

For a general viscosity coefficient $\mu\in C^{0,1}(\overline{D})$, we also have similar estimate.\\
\begin{lemma}
\label{theorem1}
$G^{\mu}(x,z_{j})$ has the estimate
\begin{small} 
\be\label{thm1}
\bigg\|G^{\mu}(x,z_{j})-\bigg\lbrace\frac{\mu(z_{\ast})}{\mu(z_{\ast})+\mu_0} \mathcal{G}^{\mu(z_{\ast})}(x-z_j)+\frac{\mu_0}{\mu(z_{\ast})+\mu_0} \mathcal{G}^{\mu_0}(x-z_j)\bigg\rbrace\bigg\|_{\bm{H}^1(D)^3}\leq C
\en
\end{small}
for all $j\in\mathbb{N}$, where $C$ is a positive constant independent of  $j\in\mathbb{N}$.
\end{lemma}

\begin{proof}
Let
\be\no
\widetilde{\mu}:=\begin{cases}
\mu(z_{\ast}) &\text{ for } x\in \overline D,\\
\mu_0 &\text{ for } x\in\Omega\backslash\overline{D}.
\end{cases}
\en
It follows from Lemma \ref{lemma1} that 
\begin{small} 
\be
\label{4.17}
\bigg\|G^{\widetilde{\mu}}(x-z_j)-\bigg\lbrace\frac{\mu(z_{\ast})}{\mu(z_{\ast})+\mu_0} \mathcal{G}^{\mu(z_{\ast})}(x-z_j)+\frac{\mu_0}{\mu(z_{\ast})+\mu_0} \mathcal{G}^{\mu_0}(x-z_j)\bigg\rbrace\bigg\|_{\bm{H}^1(D)^3}\leq C
\en
\end{small} 
for all $j\in\mathbb{N}$. Define $\widetilde{G}^{\mu}(x,z_j):=G^{\mu}(x,z_j)-G^{\widetilde{\mu}}(x,z_j)$ and $\widetilde{r}^{\mu}(x,z_j):=r^{\mu}(x,z_j)-r^{\widetilde{\mu}}(x,z_j)$.
It is easily seen from \eqref{3.0} and \eqref{lem1} that $(\widetilde{G}^{\mu}(x,z_j),\widetilde{r}^{\mu}(x,z_j))$ satisfies the following equations 
\be\label{4.18}
\begin{cases}
-{\rm{div}}(\mu(x)\mathcal{D}(\widetilde{G}_l^{\mu}(x,z_j)))+\nabla \widetilde{r}_l^{\mu}(x,z_j)=\Theta(x,z_j)& \text{ in } \Omega,\\
{\rm div}\widetilde{G}_l^{\mu}(x,z_j)=0 &\text{ in }\Omega,\\
\widetilde{G}_l^{\mu}(x,z_j)=0 &\text{ on }\partial\Omega,
\end{cases}
\en
for each $ j\in N$, $l=1,2,3$, where $\Theta(x,z_j):={\rm{div}}(\mu(x)-\mu(z_{\ast})\mathcal{D}(G_l^{\widetilde{\mu}})(x,z_j))$. By integration by parts, $\widetilde{G}_l^{\mu}(x,z_j)$ satisfies the variational equation
\begin{small}
\be\no
a(\widetilde{G}_{l}^{\mu},\phi)&:=&\int_{\Omega}
(\mu(x)\mathcal{D}(\widetilde{G}_{l}^{\mu,D}(x,z_j)):\mathcal{D}(\phi)(x))\dx\\\no
&=&\int_{D}(\widetilde{\mu}(z_{\ast})-\mu(x))\mathcal{D}(G_{l}^{\widetilde{\mu}}):\mathcal{D}(\phi)\dx-\int_{\partial D}(\widetilde{\mu}(z_{\ast})-\mu(x))\mathcal{D}(G_{l}^{\widetilde{\mu}})\bm{n}\cdot\phi\ds\\\label{4.19}
&:=&\langle \bm{\xi}^{(j)}_{1,l},\mathcal{D}(\phi)\rangle_{L^2(\Omega)^{3\times3}}+\langle\bm{\xi}^{(j)}_{2,l},\phi\rangle_{\bm{H}^{-1/2}(\partial D)\times\bm{H}^{1/2}(\partial D)}
\en
\end{small}
for all $\phi\in\bm{H}_0^1(\Omega)$ satisfying $\dive \phi=0$. 
Since $\mu(x) \in C^{0,1}(\overline{D})$, it is demonstrated by \eqref{4.17} that
\ben
\|\bm{\xi}^{(j)}_{1,l}\|_{\bm{L}^2(D)}+\|\bm{\xi}^{(j)}_{2,l}\|_{\bm{H}^{-1/2}(\partial D)}<C.
\enn
Further, applying the korn's inequality and the Lax-Milgram theorem, we obtain 
\be\label{4.20}
 \|\widetilde{G}_l^{\mu}(x-z_j)\|_{\bm{H}^1(\Omega)}<C,
 \en
for all $ j\in N$, $l=1,2,3$. The required estimate \eqref{thm1} thus follows from \eqref{4.17} and \eqref{4.20}. This completes the proof of the lemma.
\end{proof}
\subsection{ A coupled Stokes-Brinkman problem}
Let $D_0\subset\mathbb{R}^3$ be a simply connected and bounded domain with $\pa D_0\in C^2$. Consider the following coupled Stokes-Brinkman problem in $D_0$, which is modeled by
\be\label{3.4}
\begin{cases}
-{\rm div}(\mu_{a}\mathcal{D}(\bm{w}_1))+\nabla q_{1}+\chi^2\bm{w}_1=\bm{\rho} &\quad \text{ in }D_0,\\
{\rm div}\bm{w}_1=0 &\quad \text{ in }D_0,\\
-{\rm div}(\mu_b\mathcal{D}(\bm{w}_2))+\nabla q_{2}=0 &\quad\text{ in }D_0,\\
{\rm div}\bm{w}_2=0 &\quad\text{ in }D_0,\\
\bm{w}_1-\bm{w}_2=\bm{h} &\quad\text{ on }\partial D_0,\\
t(\bm{w}_1,q_{1})-t(\bm{w}_2,q_{2})=\bm{g} &\quad \text{ on }\partial D_0,
\end{cases}
\en
where $\mu_{a},\mu_{b}$ are two different positive constants, and $\bm{\rho}\in\bm{L}^2(D_0)$, $\bm{h}\in\bm{H}_0^{1/2}(\partial D_0)$, $\bm{g}\in\bm{H}^{-1/2}(\partial D_0)$.
\\

\begin{lemma}
\label{lemma2}
If ${\rm Im}\chi^2\neq 0$, there exists a unique solution $((\bm{w}_1,q_{1}),(\bm{w}_2,q_{2}))\in(\bm{H}^1(D_0)\times\bm{L}^2(D_0))\times (\bm{H}^1(D_0)\times \bm{L}^2(D_0))$ to problem \eqref{3.4}, such that
\be\label{lem2.2}
\|\bm{w}_1\|_{\bm{H}^1(D_0)}+\|\bm{w}_2\|_{\bm{H}^1(D_0)}\leq
c(\|\bm{\rho}\|_{\bm{L}^2(D_0)}+\|\bm{h}\|_{\bm{H}^{1/2}(\partial D_0)}+\|\bm{g}\|_{\bm{H}^{-1/2}(\partial D_0)}).
\en
\end{lemma}

\begin{proof}
We seek a solution of \eqref{3.4} in the form
\be\label{4.21}
&&\bm{w}_1(x)=\frac{1}{\mu_a}\bm{V}^{\mu_a,\chi^2}_{\partial D_0}(x,\bm{\phi}_1)+\frac{1}{\mu_a}\bm{W}^{\mu_a,\chi^2}_{\partial D_0}(x,\bm{\phi}_2)+\bm{V}_{D_0}^{\mu_a,\chi^2}(x,\mathbf{\rho}),\qquad\\\label{4.22}
&&\bm{w}_2(x)=\frac{1}{\mu_b}\bm{V}^{\mu_b}_{\partial D_0}(x,\bm{\phi}_1)+\frac{1}{\mu_b}\bm{W}_{\partial D_0}(x,\bm{\phi}_2),\quad x\in\mathbb{R}^3\backslash\partial D_0,\qquad
\en
and
\be
\label{4.23}
&&q_{1}(x)=\frac{1}{\mu_a} P^{s}_{\partial D_0}(x,\bm{\phi}_1)+\frac{1}{\mu_a}P_{\mu_a,\partial D_0}^{d}(x,\bm{\phi}_2)+P_{D_0}(x,\mathbf{\rho}),\qquad\\
\label{4.24}
&&q_{2}(x)=\frac{1}{\mu_b}P^{s}_{\partial D_0}(x,\bm{\phi}_1)+\frac{1}{\mu_b}P_{\mu_b,\partial D_0}^{d}(x,\bm{\phi}_2),\quad x\in\mathbb{R}^3\backslash\partial D_0.\qquad
\en
Using the properties of the layer potentials, the transmission problem \eqref{3.4} can be
reduced to the system of integral equations
\be
\label{4.25}
\bm{A} \left({\begin{array}{c}
\bm{\phi}_1  \\
\bm{\phi}_2  \\
\end{array}} \right)+\bm{B}\left({\begin{array}{c}
\bm{\phi}_1  \\
\bm{\phi}_2  \\
\end{array}} \right)=\left( {\begin{array}{c}
\bm{g}-t(\bm{V}_{D_0}^{\mu_a,\chi^2}(x,\mathbf{\rho}),P_{D_0}(x,\mathbf{\rho}))  \\
\bm{h}-\bm{V}_{D_0}^{\mu_a,\chi^2}(x,\mathbf{\rho})    \\
\end{array}} \right)
\en
in $\bm{H}^{-1/2}(\partial D_0)\times\bm{H}_0^{1/2}(\partial D_0)$, where the operators $\bm{A}$ and $\bm{B}$ are given by
\ben
\bm{A}=
\left(
\begin{array}{cc}
(\frac{1}{2\mu_a}-\frac{1}{2\mu_b})I & 0\\
0  & (\frac{1}{2\mu_b}-\frac{1}{2\mu_a})I\\
\end{array}
\right)
\enn
and
\ben
\bm{B}=
\left(
\begin{array}{cc}
\frac{1}{\mu_a}\bm{K}^{\ast,\mu_a,\chi^2}_{\partial D_0}-\frac{1}{\mu_b}\bm{K}^{\ast}_{\partial  D_0}& \frac{1}{\mu_a}\bm{D}^{\mu_a, \chi^2}_{\partial D_0}- \frac{1}{\mu_b}\bm{D}^{\mu_b}_{\partial D_0}\\
\frac{1}{\mu_a}\bm{\mathcal{V}}^{\mu_a, \chi^2}_{\partial D_0}- \frac{1}{\mu_b}\bm{\mathcal{V}}^{\mu_b}_{\partial D_0}& \frac{1}{\mu_a}\bm{K}^{\mu_a, \chi^2}_{\partial D_0}-\frac{1}{\mu_b}\bm{K}^{\mu_b}_{\partial D_0}\\
\end{array}
\right).
\enn
Here, ${I}$ denotes the identity operator in the corresponding Banach spaces. It follows from Theorems 4.1 and 4.2 in \cite{Kohr2008} that all operators in $\bm{B}$ are compact in the corresponding Banach spaces. Equation \eqref{4.25} is thus of Fredholm type with index 0. Therefore, the existence of a solution of \eqref{4.25} follows from the uniqueness of  \eqref{4.25}.

Let $(\bm{\phi}_1,\bm{\phi}_2)\in\bm{H}^{-1/2}(\partial D_0)\times\bm{H}_0^{1/2}(\partial D_0)$ be an arbitrary solution of the homogeneous equation \eqref{4.25}. 
Define $\bm{v}_1(x):=\bm{w}_1(x)-\bm{V}_{D_0}^{\mu_a,\chi^2}(x,\mathbf{\rho})$ and $p_1(x):=q_1(x)-P_{D_0}(x,\mathbf{\rho})$.
Then, $((\bm{v}_1(x),p_1(x)),(\bm{w}_2(x),q_2(x))$ satisfy the following coupling problem
 \be\label{4.26}
\begin{cases}
-{\rm div}(\mu_a\mathcal{D}(\bm{v}_1))+\nabla p_1+\chi^2\bm{v}_1=0 &\quad \text{ in }D_0,\\
{\rm div}\bm{v}_1=0,&\quad \text{ in }D_0,\\
-{\rm div}(\mu_b\mathcal{D}(\bm{w}_2))+\nabla q_2=0 &\quad\text{ in }D_0,\\
{\rm div}\bm{w}_2=0 &\quad\text{ in }D_0,\\
\bm{v}_1=\bm{w}_2 &\quad\text{ on }\partial D_0,\\
t(\bm{v}_1,p_1)=t(\bm{w}_2,q_2) &\quad \text{ on }\partial D_0.
\end{cases}
\en
By applying the Green's first identity and the boundary conditions \eqref{4.26}, we obtain
\be
\label{4.27}
\mu_a\int_{D_0}D(\bm{v}_1):D(\bm{v}_1)\dx+\chi^2 \int_{D_0}|\bm{v}_1|^2\dx=
\mu_b\int_{D_0}D(\bm{w}_2):D(\bm{w}_2)\dx.
\en
 This, together with ${\rm Im}\chi^2\neq 0$, implies $(\bm{v}_1,p_1)=(\bm{w}_2,q_2)=(0,c)$ in $D_0$. Furthermore, it can be derived from the jump property that 
\be\label{4.28}
\begin{cases}
\bm{v}_{1,+}(x)=\frac{1}{\mu_a}\bm{\phi}_2(x),\quad t(\bm{v}_{1,+},p_{1,+})(x)-c\bm{n}(x)=-\frac{1}{\mu_a}\bm{\phi}_1(x) &\quad x
\in\partial D_0,\\
\bm{w}_{2,+}(x)=\frac{1}{\mu_b}\bm{\phi}_2(x), \quad t(\bm{w}_{2,+},q_{2,+})(x)-c\bm{n}(x)=-\frac{1}{\mu_b}\bm{\phi}_1(x)&\quad x
\in\partial D_0.\\
\end{cases}
\en
Here, $\bm{n}$ denotes the outward normal at $\partial D_0$, and the superscript $\pm$ indicates the limit values of the function approaching $\partial D_0$ from the exterior and interior of $D_0$, respectively. Utilizing again the first Green's formula and \eqref{4.28}, we have
\begin{small}
\be\no
&&\mu_a^3\int_{\mathbb{R}^3\backslash\overline{D_0}}\mathcal{D}(\bm{v}_1):\mathcal{D}(\bm{v}_1)\dx+\chi^2\mu_a^2\int_{\mathbb{R}^3\backslash\overline{D_0}}|\bm{v}_1|^2\dx-\mu_b^3\int_{\mathbb{R}^3\backslash\overline{D_0}}\mathcal{D}(\bm{w}_2):\mathcal{D}(\bm{w}_2)\dx\\\no
&&=\int_{\partial D_0}\mu_b^2\bm{w}_{2,+}t(\bm{w}_{2,+},q_{2,+})-\mu_a^2\bm{v}_{1,+}t(\bm{v}_{1,+},p_{1,+})\ds\\\label{4.29}
&&=\int_{\partial D_0}c\bm{n}(\mu_b-\mu_a)\bm{\phi}_2\ds=0,
\en
\end{small}
where $\bm{\phi}_2 \in \bm{H}_0^{1/2}(\partial D_0)$. Extracting the imaginary part of \eqref{4.29} and utilizing \eqref{4.28}, it easily follows that $(\bm{v}_1,p_1)=(\bm{w}_2,q_2)=(\0,0)$ in $\mathbb{R}^3\backslash\overline{D_0}$ and $\bm{\phi}_1=\bm{\phi}_2=\0$ on $\partial D_0$, due to the uniqueness of the exterior Stokes problem. Thus, in view of Fredholm's alternative, there exists a unique solution $(\bm{\phi}_1,\bm{\phi}_2)\in\bm{H}^{-1/2}(\partial D_0)\times\bm{H}_0^{1/2}(\partial D_0)$ for \eqref{4.25}, satisfying the estimate
\ben
\|\bm{\phi}_1\|_{\bm{H}^{-1/2}(\partial D_0)}+\|\bm{\phi}_2\|_{\bm{H}_0^{1/2}(\partial D_0)}<c(\|\bm{\rho}\|_{\bm{L}^2(D_0)}+\|\bm{h}\|_{\bm{H}^{1/2}(D_0)}+\|\bm{g}\|_{\bm{H}^{-1/2}(D_0)}).
\enn
From the properties of the layer and volume potentials in \cite{Kohr2008}, the required estimate \eqref{lem2.2} follows. The proof is thus complete.
\end{proof}

\section{Proof of Theorem \ref{theorem2}}
This section is devoted to a detailed  proof of Theorem \ref{theorem2}. A novel and simple  technique will be proposed, based on a priori uniform estimate of the Dirichlet Green's functions and the well-posedness of a coupled Stokes-Brinkman equation in the $\bm{H}^1$-space. In the following, we intend to prove the global uniqueness of $\mu$ by four claims.

{\bf claim 1 }$\mu_{1}|_{\Omega\backslash\overline{D}_1}=\mu_{2}|_{\Omega\backslash\overline{D}_2}$.

This will be proved by contradiction. Assume $\mu_{1}|_{\Omega\backslash\overline{D}_1}\neq\mu_{2}|_{\Omega\backslash\overline{D}_2}$. Without loss of generality, for any point  $y_{\ast}\in \pa \Omega$, we can select a simply connected, $C^2$-smooth domain $\Omega_{0}\subseteq (\Omega\backslash\overline{D}_i)$ for $i=1,2$, and such that $y_{\ast}\in \pa \Omega_{0}\cap \pa\Omega $ which contains a non-empty open subset in the 2-dimensional manifold of $\R^3$. Define
\be\label{y}
y_j:=y_{\ast}+(\delta/j)\bm{n}(y_{\ast}), j=1,2,...
\en
with a sufficiently small $\delta > 0$, where $\bm{n}$ is the unit outward normal to $\partial \Omega$. 

Let $(\bm{u}^{(j)}_{1,l}(x),p^{(j)}_{1,l}(x))$ and $(\bm{u}^{(j)}_{2,l}(x),p^{(j)}_{2,l}(x))$ be the solutions to (1.1) with the boundary data $\bm{f}^{(j)}_l:=\mathcal{G}^{\mu_{0,1}}(x,y_j)e_l$, corresponding to $\mu_1(x)$ and $\mu_2(x)$ respectively, where $\mu_{0,1}:=\mu_1|_{\Omega\backslash\overline{D_1}}$. Define $\widetilde{\bm{u}}^{(j)}_{1,l}(x):=\bm{u}^{(j)}_{1,l}(x)-\mathcal{G}^{\mu_{0,1}}(x,y_j)e_l$ and $\widetilde{p}^{(j)}_{1,l}(x):=p^{(j)}_{1,l}(x)-\Pi(x,y_j)e_l$. Obviously,
 $(\widetilde{\bm{u}}^{(j)}_{1,l}(x),\widetilde{p}^{(j)}_{1,l}(x))$ satisfies the following Stokes problem 
\ben
\begin{cases}
-{\rm{div}}(\mu_{0,1}\mathcal{D}(\widetilde{\bm{u}}^{(j)}_{1,l}(x))+\nabla \widetilde{p}^{(j)}_{1,l}(x)=0& \text{ in } \Omega\backslash\overline{D_1},\\
{\rm div}\widetilde{\bm{u}}^{(j)}_{1,l}(x)=0 &\text{ in }\Omega\backslash\overline{D_1},\\
\widetilde{\bm{u}}^{(j)}_{1,l}(x)=0&\text{ on }\partial\Omega,\\
\widetilde{\bm{u}}^{(j)}_{1,l}(x)=\bm{u}^{(j)}_{1,l}(x)-\mathcal{G}^{\mu_{0,1}}(x,y_j)e_l&\text{ on }\partial D_1,
\end{cases}
\enn
for each $ j\in N$, $l=1,2,3$. It is easily seen that
\be\label{+4.03}
\|\bm{u}^{(j)}_{1,l}(x)-\mathcal{G}^{\mu_{0,1}}(x,y_j)e_l\|_{\bm{H}^{1}(\Omega\backslash\overline{D_1})}\leq C
\en
for all $j\in N$, $l=1,2,3$, due to the positive distance between $\partial D_1$ and $y_j$. Moreover, we can easily check that $(\bm{u}^{(j)}_{1,l}(x),p^{(j)}_{1,l}(x)),(\bm{u}^{(j)}_{2,l}(x),p^{(j)}_{2,l}(x))$ satisfy the coupled Stokes-Brinkman problem \eqref{3.4} in $\Omega_0$ with
\ben
\begin{cases}
\bm{\rho}^{(j)}_l:=\chi^2\bm{u}^{(j)}_{1,l}(x),\\
\bm{f}^{(j)}_l:=\bm{u}^{(j)}_{1,l}(x)-\bm{u}^{(j)}_{2,l}(x),\\
\bm{g}^{(j)}_l:=t(\bm{u}^{(j)}_{1,l},p^{(j)}_{1,l})(x)-t(\bm{u}^{(j)}_{2,l},p^{(j)}_{2,l})(x),
\end{cases}
\enn
and $\mu_a:=\mu_{1}|_{\Omega\backslash\overline{D}_1}, \mu_b:=\mu_{2}|_{\Omega\backslash\overline{D}_2}$. Due to $\mathcal{C}_{\mu_1}=\mathcal{C}_{\mu_2}$ and \eqref{+4.03}, one has
\ben
\|\bm{\rho}^{(j)}_l\|_{\bm{L}^{2}(\Omega_0)}+
\|\bm{f}^{(j)}_l\|_{\bm{H}^{1/2}(\partial \Omega_0)}+\|\bm{g}^{(j)}_l\|_{\bm{H}^{-1/2}(\partial \Omega_0)}
\leq &C.
\enn
Then, it follows from Lemma \ref{lemma2} with \eqref{+4.03} that
\ben
\|\mathcal{G}^{\mu_{0,1}}(x,y_j)e_l\|_{\bm{H}^{1}(\Omega_0)}\leq \|\bm{u}^{(j)}_{1,l}(x)-\mathcal{G}^{\mu_{0,1}}(x,y_j)e_l\|_{\bm{H}^{1}(\Omega_0)}+\|\bm{u}^{(j)}_{1,l}(x)\|_{\bm{H}^{1}(\Omega_0)}\leq C
\enn
for all $j\in N$, $l=1,2,3$. However, this leads to a contradiction since the fact
\ben
\|\mathcal{G}^{\mu_{0,1}}(x,y_j)\|_{\bm{H}^{1}(\Omega_0)}\rightarrow \infty\quad as \ j\rightarrow \infty.
\enn
Hence, $\mu_1|_{\Omega\backslash\overline{D_1}}=\mu_2|_{\Omega\backslash\overline{D_2}}$.

{\bf claim 2 } $D_1=D_2$.

For any $\bm{f}\in \bm{H}_0^{\frac{1}{2}}(\partial\Omega)$, let $(\bm{u}_1(x),p_1(x))$, $(\bm{u}_2(x),p_2(x))$ be the solutions to \eqref{model1} with the boundary data $\bm{f}$ and $(G^{\mu_1}(x,z),r^{\mu_1}(x,z)$,$(G^{\mu_2}(x,z),r^{\mu_2}(x,z)$ be the Dirichlet Green's functions, corresponding to $\mu_1(x)$ and $\mu_2(x)$ respectively. 

By the Green's representation theorem, we obtain
\begin{small}
\be\label{4.30}
&&\bm{u}_1(x)=\int_{\partial\Omega}\left(G^{\mu_1}(x,y)t(\bm{u}_1,p_1)(y)
-t'_{y}(G^{\mu_1},r^{\mu_1})(x,y)^{\top}\bm{u}_1(y)\right)\ds(y)\quad{\text{for}}\ x\in \Omega, \qquad\\\label{4.31}
&&\bm{u}_2(x)=\int_{\partial\Omega}\left(G^{\mu_2}(x,y)t(\bm{u}_2,p_2)(y)
-t'_{y}(G^{\mu_2},r^{\mu_2})(x,y)^{\top}\bm{u}_2(y)\right)\ds(y)\quad{\text{for}}\ x\in \Omega, \qquad
\en
\end{small}
where $t'_{ik}(\bm{u},p):=\delta^{k}_{i}p\bm{n}
+\mu\left(\frac{\partial \bm{u}^{i}}{\partial x_k}+\frac{\partial \bm{u}^{k}}{\partial x_i}\right)\bm{n}$ and the subscript $y$ on $t'$ indicates that the differentiation in $t'$ is performed with respect to $y$. Recalling 
$\mathcal{C}_{\mu_1}=\mathcal{C}_{\mu_2}$, the unique continuation property stated in \cite{Alvarez2005} can be thus applied to deduce 
\ben
 \bm{u}_1(x)=\bm{u}_2(x) \quad x\in \Omega\cap \mathcal{U},
 \enn
where $\mathcal{U}$ denotes the unbounded component of $\mathbb{R}^3\backslash(\overline{D_1\cup D_2})$. This, combined with \eqref{4.30}-\eqref{4.31} and $\mathcal{C}_{\mu_1}=\mathcal{C}_{\mu_2}$, yields 
 \begin{small}
\ben
\int_{\partial\Omega}t'_{y}(G^{\mu_1},r^{\mu_1})(x,y)^{\top}\bm{f}(y)\ds(y)
=\int_{\partial\Omega}t'_{y}(G^{\mu_2},r^{\mu_2})(x,y)^{\top}\bm{f}(y)\ds(y)
\enn
\end{small}
for $ x\in\Omega\cap \mathcal{U}$ and all $\bm{f}\in \bm{H}_0^{\frac{1}{2}}(\partial\Omega)$, which means that 
\begin{small}
\be
\label{4.32}
t'_{y}(G^{\mu_1},r^{\mu_1})(x,y)=t'_{y}(G^{\mu_2},r^{\mu_2})(x,y),\quad\ y\in\partial\Omega, x\in\Omega\cap \mathcal{U}.
\en
\end{small}
Moreover, applying the unique continuation property once again, one has
\begin{small}
\be
\label{4.34}
G^{\mu_1}(x,y)=G^{\mu_2}(x,y),\quad
r^{\mu_1}(x,y)=r^{\mu_2}(x,y), \quad\ x,y\in\Omega\cap \mathcal{U}(x\neq y).
\en
\end{small}

In the following, we prove $D_1=D_2$ by contradiction. Suppose $D_1\neq D_2$. Without loss of generality, let $z_{\ast}\in \pa D_1 \cap \pa \mathcal{U}$ and $D_0$ be a small $C^2$-smooth domain satisfying: $\rm{\romannumeral 1)}$ $z_{\ast}\in\pa D_0 \cap \pa D_1$; $\rm{\romannumeral 2)}$ $D_0\subseteq D_1 \setminus \overline{D_2}$; $\rm{\romannumeral 3)}$ $\pa D_{0}\cap \pa D_1$ contains a non-empty open subset in the 2-dimensional manifold of $\R^3$. Define 
\ben
z_j:=z_{\ast}+(\delta/j)\bm{n}(z_{\ast}), \quad j=1,2,...,
\enn
where $\delta > 0$ is sufficiently small, and $\bm{n}$ is the unit outward normal to $\partial D_1$. Let
\be\no
\widetilde{\mu}_1(x):=\begin{cases}
\mu_1(z_{\ast}) &\text{ for } x\in \overline D,\\
\mu_0 &\text{ for } x\in\Omega\backslash\overline{D}.
\end{cases}
\en 
 $(G_l^{\widetilde{\mu}_1}(x,z_j),r^{\widetilde{\mu}_1}(x,z_j))$ and $(G^{\mu_2}(x,z_j),r^{\mu_2}(x,z_j))$ satisfy the coupled Stokes-Brinkman problem \eqref{3.4} with
\ben
\begin{cases}
\bm{\rho}^{(j)}_l:=\chi^2 G_l^{\widetilde{\mu}_1}(x,z_j),\\
\bm{h}^{(j)}_l:=G_l^{\widetilde{\mu}_1}(x,z_j)-G_l^{\mu_2}(x,z_j),\\
\bm{g}^{(j)}_l:=t(G_l^{\widetilde{\mu}_1},r_l^{\widetilde{\mu}_1})(x,z_j)-t(G_l^{\mu_2},r_l^{\mu_2})(x,z_j),
\end{cases}
\enn
and $\mu_a:=\mu_1(z_{\ast}), \mu_b:=\mu_0$ for each $j \in \mathbb{N}$, $l=1,2,3$. It then follows from \eqref{lem1} and \eqref{4.34} that
\ben
\|\bm{\rho}^{(j)}_l\|_{\bm{L}^2(D_0)}+\|\bm{h}^{(j)}_l\|_{\bm{H}^{1/2}(\partial D_0)}+\|\bm{g}^{(j)}_l\|_{\bm{H}^{-1/2}(\partial D_0)}\leq C.
\enn
Furthermore, one has by Lemma \ref{lemma2} 
\begin{small}
\ben
&&\bigg\|\frac{\mu_1(z_{\ast})}{\mu_1(z_{\ast})+\mu_0} \mathcal{G}^{\mu_1(z_{\ast})}(x-z_j)+\frac{\mu_0}{\mu_1(z_{\ast})+\mu_0} \mathcal{G}^{\mu_0}(x-z_j)\bigg\|_{\bm{H}^1(D_0)^3}\\
&\leq & \bigg\|G^{\widetilde{\mu}_1,D_1}(x,z_j)-\bigg\lbrace\frac{\mu_1(z_{\ast})}{\mu_1(z_{\ast})+\mu_0} \mathcal{G}^{\mu_1(z_{\ast})}(x-z_j)+\frac{\mu_0}{\mu_1(z_{\ast})+\mu_0} \mathcal{G}^{\mu_0}(x-z_j)\bigg\rbrace\bigg\|_{\bm{H}^1(D_0)^3}\\
&&+\|G^{\widetilde{\mu}_1,D_1}(x,z_j)\|_{\bm{H}^1(D_0)^3}\\
&\leq & C.
\enn
\end{small}
 However, this leads to a contradiction with the following fact
\begin{small}
 \be\label{4.9-}
\bigg\|\frac{\mu_1(z_{\ast})}{\mu_1(z_{\ast})+\mu_0} \mathcal{G}^{\mu_1(z_{\ast})}(x-z_j)+\frac{\mu_0}{\mu_1(z_{\ast})+\mu_0} \mathcal{G}^{\mu_0}(x-z_j)\bigg\|_{\bm{H}^1(D_0)^3}\rightarrow\infty
\en
\end{small}
as $j\rightarrow\infty$. Hence, we conclude that $D_1=D_2$.
 
{\bf claim 3 } $\mu_1(x)|_{\pa D}=\mu_2(x)|_{\pa D}$.

Following similar argument as the previous claim 1, the aim of this point is to show $\mu_1(x)|_{\pa D}=\mu_2(x)|_{\pa D}$ by contradiction. Suppose $\mu_1(x)|_{\pa D}\neq\mu_2(x)|_{\pa D}$, Clearly, we can select a point $z_{\ast}\in \pa D$ such that $\mu_1(z_{\ast})\neq \mu_2(z_{\ast})$ and a small $C^2$-smooth domain $D_0$ satisfying: $\rm{\romannumeral 1)}$ $z_{\ast}\in\pa D_0 \cap \pa D$; $\rm{\romannumeral 2)}$ $D_0\subseteq D$; $\rm{\romannumeral 3)}$ $\pa D_{0}\cap \pa D$ contains a non-empty open subset in the 2-dimensional manifold of $\R^3$. Define 
\ben
z_j:=z_{\ast}+(\delta/j)\bm{n}(z_{\ast}), j=1,2,...
\enn
with $\delta > 0$ is sufficiently small, where $\bm{n}$ denotes the unit outward normal to $\partial D$. 
Let $(G^{\widetilde{\mu}_1}(x,z_j),r^{\widetilde{\mu}_1}(x,z_j))$ and $(G^{\widetilde{\mu}_2}(x,z_j),r^{\widetilde{\mu}_2}(x,z_j))$ be the Dirichlet Green's functions corresponding to $\widetilde{\mu}_1$ and $\widetilde{\mu}_2$, respectively, where
 \be\no
\widetilde{\mu}_i(x):=\begin{cases}
\mu_i(z_{\ast}) &\text{ for } x\in \overline D,\\
\mu_0 &\text{ for } x\in\Omega\backslash\overline{D},
\end{cases}
\en
for $i=1,2$. $(G_l^{\widetilde{\mu}_1}(x,z_j),r^{\widetilde{\mu}_1}(x,z_j))$ and $(G_l^{\widetilde{\mu}_2}(x,z_j),r^{\widetilde{\mu}_2}(x,z_j))$ satisfy the coupled Stokes-Brinkman problem \eqref{3.4} with
\ben
\begin{cases}
\bm{\rho}^{(j)}_l:=\chi^2G_l^{\widetilde{\mu}_1}(x,z_j),\\
\bm{h}^{(j)}_l:=G_l^{\widetilde{\mu}_1}(x,z_j)-G_l^{\widetilde{\mu}_2}(x,z_j),\\
\bm{g}^{(j)}_l:=t(G_l^{\widetilde{\mu}_1},r_l^{\widetilde{\mu}_1})(x,z_j)-t(G_l^{\widetilde{\mu}_2},r_l^{\widetilde{\mu}_2})(x,z_j).
\end{cases}
\enn
In view of Lemma \ref{lemma1} with \eqref{4.20} and \eqref{4.34}, we obtain
\ben
\|\bm{\rho}^{(j)}_l\|_{\bm{L}^2(D_0)}+\|\bm{h}^{(j)}_l\|_{\bm{H}^{1/2}(\partial D_0)}+\|\bm{g}^{(j)}_l\|_{\bm{H}^{-1/2}(\partial D_0)}\leq C,
\enn
for all $j \in \mathbb{N}$, $l=1,2,3$. By Lemma \ref{lemma2}, one has
\ben
 \|G^{\widetilde{\mu}_1}(x,z_j)\|_{\bm{H}^1(D_0)^3}\leq C.
\enn
It then follows from Lemma \ref{lemma1} that
\begin{small}
\ben
 \bigg\|\frac{\mu_1(z_{\ast})}{\mu_1(z_{\ast})+\mu_0} \mathcal{G}^{\mu_1(z_{\ast})}(x-z_j)+\frac{\mu_0}{\mu_1(z_{\ast})+\mu_0} \mathcal{G}^{\mu_0}(x-z_j)\bigg\|_{\bm{H}^1(D_0)^3} \leq  C.
\enn
\end{small}
 This is a  contradiction due to the fact \eqref{4.9-}. Therefore, we conclude that $\mu_1(x)|_{\pa D}=\mu_2(x)|_{\pa D}$.

{\bf claim 4 } $\partial^{\alpha}\mu_1(x)|_{\pa D}=\partial^{\alpha}\mu_2(x)|_{\pa D},\quad \forall x \in \pa D, |\alpha|=1$

In the following, we shall show $\partial^{\alpha}\mu_1(x)|_{\pa D}=\partial^{\alpha}\mu_2(x)|_{\pa D}$ for all $|\alpha|=1$ by contradiction. Suppose there exists a point $z_{\ast}\in \pa D$ with $\partial^{\alpha}\mu_1(z_{\ast})\neq\partial^{\alpha}\mu_2(z_{\ast})$ for some $|\alpha|=1$. This implies $\partial_{n}\mu_1(z_{\ast})\neq\partial_{n}\mu_2(z_{\ast})$ due to $\mu_1(x)|_{\pa D}=\mu_2(x)|_{\pa D}$. We can select a sufficiently small $C^2$-smooth domain $D_0$ and define $z_{j}$ similarly as in claim 3. $(G^{\mu_1}(x,z_j),r^{\mu_1}(x,z_j))$ and $(G^{\mu_2}(x,z_j),r^{\mu_2}(x,z_j))$ satisfy the following coupled Stokes system 
\ben
\begin{cases}
-{\rm div}(\mu_1\mathcal{D}(G_l^{\mu_1}(x,z_j))+\nabla r_l^{\mu_1}(x,z_j)=0 &\quad \text{ in }D_0,\\
{\rm div}G_l^{\mu_1}(x,z_j)=0 &\quad \text{ in }D_0,\\
-{\rm div}(\mu_2\mathcal{D}(G_l^{\mu_2}(x,z_j))+\nabla r_l^{\mu_2}(x,z_j)=0 &\quad\text{ in }D_0,\\
{\rm div}G_l^{\mu_2}(x,z_j)=0 &\quad\text{ in }D_0,\\
G_l^{\mu_1}(x,z_j)-G_l^{\mu_2}(x,z_j)=0 &\quad\text{ on }\partial D_0\cap \pa D,\\
t(G_l^{\mu_1},r_l^{\mu_1})-t(G_l^{\mu_2},r_l^{\mu_2})=0 &\quad \text{ on }\partial D_0 \cap \pa D,
\end{cases}
\enn
for each $j \in \mathbb{N}$ and $l=1,2,3$. Integrating by parts in $D_0$, it then follows from Lemma \ref{theorem1} with \eqref{4.34} that
\begin{small}
\ben
&&\bigg\vert\int_{D_0}(\mu_1(x)-\mu_2(x))\mathcal{D}(G_l^{\mu_1}(x,z_j)):\mathcal{D}(G_l^{\mu_2}(x,z_j))\dx\bigg\vert\\
&=&\bigg\vert\int_{\partial D_0\setminus \pa D}\lbrace G_l^{\mu_2}(x,z_j)\cdot t(G_l^{\mu_1},r_l^{\mu_1})(x,z_j)-G_l^{\mu_1}(x,z_j)\cdot t(G_l^{\mu_2},r_l^{\mu_2})(x,z_j)\rbrace\ds(x)\bigg\vert\\
&<&C
\enn
\end{small}
for all $j \in \mathbb{N}$, $l=1,2,3$. Lemma \ref{theorem1} can be again applied to derive the estimate
\be\label{4.35}
\int_{D_0}\vert\mu_1(x)-\mu_2(x)\vert\vert\mathcal{D}(\mathcal{G}_l^{\ast}(x-z_j))\vert^2\dx<C
\en 
for all $j \in \mathbb{N}$, $l=1,2,3$, where 
\begin{small}
\ben
\mathcal{G}_l^{\ast}(x-z_j):=\frac{\mu_1(z_{\ast})}{\mu_1(z_{\ast})+\mu_0} \mathcal{G}_l^{\mu_1(z_{\ast})}(x-z_j)+\frac{\mu_0}{\mu_1(z_{\ast})+\mu_0} \mathcal{G}_l^{\mu_0}(x-z_j)).
\enn
\end{small}
However, if $\partial_{n}\mu_1(z_{\ast})\neq\partial_{n}\mu_2(z_{\ast})$, the integral in \eqref{4.35} will explode as  $j\rightarrow \infty$, because
\ben
\vert\mu_1(x)-\mu_2(x)\vert=\vert\left( \partial_{n}\mu_1(z_{\ast}) -\partial_{n}\mu_2(z_{\ast})\right)n(z_{\ast})\cdot(x-z_{\ast})\vert+\mathcal{O}\vert x-z_{\ast}\vert^2 
\enn
as $x\rightarrow z_{\ast}$, with $x\in D_0$. Hence, $\partial^{\alpha}\mu_1(x)|_{\pa D}=\partial^{\alpha}\mu_2(x)|_{\pa D}$ for all $x \in \pa D, |\alpha|=1$.

{\bf claim 5 } $\mu_1(x)|_{\overline{D}}=\mu_2(x)|_{\overline{D}}$ if $\mu_1(x),\mu_2(x) \in C^{n_0}(\overline{D})$ for $n_0\geq 8$.

To accomplish this, we first introduce an operator $L_{\mu}:\bm{H}_0^{1/2}(\partial \Omega) \rightarrow  \bm{H}_0^{1/2}(\partial D)$, defined by $L_{\mu}\bm{f}=\bm{u}|_{\partial D} $, where $\bm{u}$ is the solution of the Stokes equations \eqref{model1} with boundary data $\bm{f}$.
Our goal is to demonstrate the denseness of Range$(L_{\mu})$ in $\bm{H}_0^{1/2}(\partial D)$. Following from the knowledge of functional analysis, it suffices to show the injectivity of the adjoint operator $L_{\mu}^{\ast}:(\bm{H}_0^{1/2}(\partial D))^{\ast}\rightarrow (\bm{H}_0^{1/2}(\partial \Omega))^{\ast}$. To this end, we consider the following problem
\be\label{+model1}
\begin{cases}
-{\rm div}(\mu_0\mathcal{D}(\bm{v}))+\nabla m=0,&\quad \text{ in }\Omega\backslash\overline{D},\\
{\rm div}\bm{v}=0,&\quad \text{ in }\Omega\backslash\overline{D},\\
-{\rm div}(\mu(x)\mathcal{D}(\bm{v}))+\nabla m=0,&\quad\text{ in }D,\\
{\rm div}\bm{v}=0,&\quad\text{ in }D,\\
\bm{v}_+=\bm{v}_-,&\quad\text{ on }\partial D,\\
t(\bm{v}_+,m_+)-t(\bm{v}_-,m_-)=\overline{\bm{\varsigma}},&\quad \text{ on }\partial D,\\
\bm{v}=0,&\quad\text{ on }\partial \Omega,
\end{cases}
\en
with a given $\bm{\varsigma} \in (\bm{H}_0^{1/2}(\partial D))^{\ast}$ in the distribution sense. Here, the superscript $\bm{v}_\pm$ and $m_\pm$ indicate the limit values of $\bm{v}$ and $m$ approaching $\partial D$ from the exterior and interior of $D$, respectively. By applying the variational method, it can be verified that problem \eqref{+model1} is well-posed for any $\varsigma \in (\bm{H}_0^{1/2}(\partial D))^{\ast}$. Using Green?s formula in conjunction with \eqref{+model1}, we derive 
\ben
\langle L_{\mu}\bm{f},\bm{\varsigma} \rangle_{\bm{H}_0^{1/2}(\partial D)\times(\bm{H}_0^{1/2}(\partial D))^{\ast}}&=&\int_{\partial D}\bm{u}\cdot \lbrace t(\bm{v}_+,m_+)-t(\bm{v}_-,m_-)\rbrace \ds\\
&=&\int_{\partial \Omega}\bm{u}\cdot t(\bm{v},m)\ds,
\enn
which implies $L_{\mu}^{\ast}{\bm{\varsigma}}=t(\overline {\bm{v}},\overline{m})$.
Suppose $L_{\mu}^{\ast}\bm{\varsigma}=t(\overline {\bm{v}},\overline{m})=0$ for some $\varsigma \in (\bm{H}_0^{1/2}(\partial D))^{\ast}$. The unique continuation property can be applied to deduce $\bm{v}=0$ in $\Omega\backslash\overline{D}$. This, along with the well-posedness of the Dirichlet-Stokes problem, implies $\bm{v}=0$ in $D$. It then follows from the boundary condition of $\bm{v}$ on $\partial D$ that $\bm{\varsigma} = 0$.  Hence, the Range($L_{\mu}$) is dense in $\bm{H}_0^{1/2}(\partial D)$. 
Next, consider the Dirichlet Stokes problem given by
\be\label{+model2}
\begin{cases}
-{\rm div}(\mu(x)\mathcal{D}(\bm{u}))+\nabla p=0,&\quad\text{ in }D,\\
{\rm div}\bm{u}=0,&\quad\text{ in }D,\\
\bm{u}=\bm{\tau},&\quad\text{ on }\partial D,
\end{cases}
\en
for any $\bm{\tau} \in \bm{H}^{1/2}_0(\partial D)$. Define the corresponding Dirichlet-to-Neumann map $\Lambda_{\pa D,\mu}:\bm{H}^{1/2}_0(\partial D)\rightarrow \bm{H}^{-1/2}(\partial D)$ and the Cauchy data set 
\ben
\mathcal{C}_{\pa D,\mu}:=\{(\bm{\tau},\Lambda_{\pa D,\mu}(\bm{\tau})):\bm{\tau}\in \bm{H}_0^{1/2}(\partial D)\}.
\enn
For any $\bm{\tau} \in \bm{H}^{1/2}_0(\partial D)$, there exists a uniformly bounded sequences ${\bm{f}}^{(j)} \in \bm{H}_0^{1/2}(\partial \Omega)$ satisfying $\lim_{j\to \infty}\|L_{\mu_1}\bm{f}^{(j)}-\bm{\tau}\|_{\bm{H}^{1/2}(\pa D)}=0$. From the fact that $\mathcal{C}_{\mu_1}=\mathcal{C}_{\mu_2}$ and $\mu_{1}|_{\Omega\backslash\overline{D}}=\mu_{2}|_{\Omega\backslash\overline{D}}$, we have
\ben
L_{\mu_{1}}\bm{f}^{(j)}=L_{\mu_{2}}\bm{f}^{(j)},\quad
t(\bm{u}^{(j)}_1,p^{(j)}_1)|_{\pa D}=t(\bm{u}^{(j)}_2,p^{(j)}_2)|_{\pa D}, 
\enn
where $(\bm{u}^{(j)}_1(x),p^{(j)}_1(x))$ and $(\bm{u}^{(j)}_2(x),p^{(j)}_2(x))$ are the solutions to the Stokes problem  \eqref{model1} with the boundary data $\bm{f}^{(j)}$, corresponding to $\mu_1(x)$ and $\mu_2(x)$ respectively. This leads to 
\ben
\Lambda_{\pa D,\mu_1}(\bm{\tau})=\lim\nolimits_{j \to \infty} t(\bm{u}^{(j)}_1,p^{(j)}_1)|_{\pa D}=\lim\nolimits_{j \to \infty} t(\bm{u}^{(j)}_2,p^{(j)}_2)|_{\pa D}=\Lambda_{\pa D,\mu_2}(\bm{\tau}),
\enn
which implies that $\mathcal{C}_{\pa D,\mu_1}=\mathcal{C}_{\pa D,\mu_2}$. Finally, combining with claim 3 and claim 4, it is following from \cite[Theorem 1]{Li2007} that $\mu_1(x)|_{\overline{D}}=\mu_2(x)|_{\overline{D}}$. The proof is thus complete.\\

\section*{Acknowledgments}

{This work was supported by the NNSF of China under Grand No. 12571460.}

\bibliographystyle{siam}
\newpage
\bibliography{ref3}

\end{document}